\documentclass[reqno]{amsart}

\usepackage{times}
\usepackage{mathrsfs}
\usepackage{amsmath}
\usepackage{amsthm}
\usepackage{amsfonts}
\usepackage{color}
\usepackage{graphicx,psfrag,epsfig}

\newtheorem{thm}{Theorem}[section]
\newtheorem{lem}[thm]{Lemma}
\newtheorem{cor}[thm]{Corollary}
\newtheorem{prop}[thm]{Proposition}

\DeclareMathAlphabet{\mathpzc}{OT1}{pzc}{m}{it}

\numberwithin{equation}{section}

\newcommand{\R}{\mathbb{R}}

\newcommand{\ve}{\varepsilon}
\newcommand{\rd}{\mathrm{d}}

\newcommand{\beqn}{\begin{equation}}
\newcommand{\eeqn}{\end{equation}}
\newcommand{\bqnn}{\begin{equation*}}
\newcommand{\eqnn}{\end{equation*}}
\newcommand{\bear}{\begin{eqnarray}} 
\newcommand{\eear}{\end{eqnarray}} 
\newcommand{\bean}{\begin{eqnarray*}} 
\newcommand{\eean}{\end{eqnarray*}} 
\newcommand{\bs}{\begin{split}}
\newcommand{\es}{\end{split}}

\newcommand{\eps}{\varepsilon}

\newcommand{\dhr}{\mathrel{\lhook\joinrel\relbar\kern-.8ex\joinrel\lhook\joinrel\rightarrow}}

\title[Thin Film Equations with Soluble Surfactant]
{Thin Film Equations with Soluble Surfactant and Gravity: Modeling and Stability of Steady States \footnote{Partially supported by the french-german PROCOPE project 20190SE}}

\author[Escher, Hillairet, Lauren\c{c}ot, Walker]{Joachim Escher, Matthieu Hillairet, Philippe Lauren\c{c}ot, Christoph Walker}

\address{Institut de Math\'{e}matiques de Toulouse, CNRS UMR~5219, Universit\'{e} de Toulouse, F--31062 Toulouse cedex 9, France.}
\email{matthieu.hillairet@math.univ-toulouse.fr}
\email{philippe.laurencot@math.univ-toulouse.fr}

\address{Leibniz Universit\"at Hannover, Institut f\"ur Angewandte Mathematik, Welfengarten 1, D--30167 Hannover, Germany.}
\email{escher@ifam.uni-hannover.de}
\email{walker@ifam.uni-hannover.de}

\begin{document}

\begin{abstract}
A thin film on a horizontal solid substrate and coated with a soluble surfactant is considered. The governing degenerate parabolic equations for the film height and the surfactant concentrations on the surface and in the bulk are derived using a lubrication approximation when gravity is taken into account. It is shown that the steady states are asymptotically stable.
\end{abstract}


\maketitle

\section{Introduction}

The dynamics of a thin liquid film on a solid substrate has attracted considerable interest in the past, allowing many applications e.g. in physics, engineering, or biomedicine (see \cite{Grotberg, GrotbergJensen92, GrotbergJensen93, LinHwang00, LinHwang02, WarnerCraster} and the references therein). When such a film is coated with a distribution of surfactant, i.e. surface active agents that lower the surface tension of the liquid, the resulting gradients of surface tension induce so-called Marangoni stresses (e.g. \cite{WarnerCraster}) that cause the surfactant to spread. From the point of view of potential applications, estimates on the velocity of the spreading front and possible thinning and rupture of the film are among the main issues under investigation, see \cite{LinHwang00} and the references therein. While most part of the research hitherto has been dedicated to the case of insoluble surfactant (e.g., see \cite{deWitGallez, SharmaRuckenstein86} and, for further references, see \cite{LinHwang00, LinHwang02}), this article focuses on the actually more common case of soluble surfactants when there are sorptive fluxes between the surface and the bulk region immediately beneath the surface \cite{GrotbergJensen93, LinHwang00, LinHwang02, Matar02, SharmaRuckenstein86, WarnerCraster}. Sorption seems to play an important role during the early stages of the spreading. For instance, soluble surfactants may restrain the spreading of liquid films \cite{GrotbergJensen93} or delay film rupture \cite{SharmaRuckenstein86}. Besides including sorptive fluxes we take into account also the effect of gravity but neglect effects of capillarity and van der Waals forces.
Based on the full free boundary problem for the Navier-Stokes equations and the mass balance equations governing the dynamics of the thin liquid and the surfactant concentrations, respectively, we derive a system of degenerate parabolic equations for these quantities by using a lubrication approximation as, e.g., in \cite{GrotbergJensen92, GrotbergJensen93}. 
Cross-sectional averaging then allows one to neglect vertical components of the bulk surfactant concentration so that the resulting coupled system consists of one-dimensional equations. 

In dimensionless form the system reads:
\begin{align}
&\partial_t h-\partial_x\left(\frac{G}{3}\,h^3\,\partial_x h-\frac{1}{2}\, h^2\,\partial_x\sigma(\Gamma)\right)=0\ ,\label{1}\\
&\partial_t(h\,C_0)-\partial_x\left(\frac{G}{3}\,h^3\, C_0\,\partial_x h-\frac{1}{2}\,h^2\,C_0\,\partial_x\sigma(\Gamma)+\delta\, h\,\partial_x C_0\right)=-\beta\, K\, (C_0-\Gamma)\ ,\label{2}\\
&\partial_t\Gamma-\partial_x\left(\frac{G}{2}\,h^2\,\Gamma\,\partial_xh-h\,\Gamma\,\partial_x\sigma(\Gamma)+D\,\partial_x\Gamma\right)=K\, (C_0-\Gamma)\label{3}
\end{align}
for $x\in (0,L)$ and $t>0$ subject to suitable initial conditions at $t=0$:
\beqn\label{4a}
h(0,\cdot)=h^0\ ,\quad C_0(0,\cdot)=C_0^0\ ,\quad \Gamma(0,\cdot)=\Gamma^0\ ,
\eeqn
and no-slip boundary conditions at $x=0$ and $x=L$:
\beqn\label{4}
\partial_x h=\partial_x C_0=\partial_x\Gamma=0\ ,\quad x=0, L\ .
\eeqn
Here, $h$ denotes the film thickness, $\Gamma$ denotes the surface surfactant concentration, and $C_0$ is the (cross-sectionally averaged) bulk surfactant concentration. The constant $\beta>0$ indicates the degree of solubility so that in the limit $\beta\rightarrow 0$ the surfactant is highly soluble in the substrate and adsorbs  weakly on the surface, while $\beta\rightarrow\infty$ corresponds to a surfactant accumulating preferably on the surface with weak bulk solubility. The surface tension $\sigma(\Gamma)$ is a decreasing function of the surface surfactant concentration $\Gamma$. The positive constants $D$ and $\delta$ measure molecular diffusion of surfactants on the surface and in the bulk, respectively. For more information on the model we refer to Section~\ref{sec2}.

While the derivation of models, both for insoluble and soluble surfactant, can be found in many research papers (the latter, however, merely in the situation when gravity is neglected), there seems to be only few literature investigating thin film equations with surfactant from a mathematical analytical point of view. In addition, to the best of our knowledge all research so far is dedicated to the insoluble surfactant case, i.e. when the bulk surfactant concentration $C_0$ in equations \eqref{1}-\eqref{3} (and possibly other terms) is neglected. For instance, in \cite{Renardy96a, Renardy96b, Renardy97} local existence results can be found for the resulting systems of degenerate equations for $h$ and $\Gamma$. In \cite{GarckeWieland} global existence of weak solutions is shown for a variant of \eqref{1}-\eqref{3} (with $C_0$ neglected) including fourth order terms in $h$ modeling capillarity effects. This fourth order degenerate parabolic thin film equation for $h$
has been rigorously derived recently from Stokes flow with surface tension not using a lubrication approximation,
cf. \cite{GP08}. Thin films on a inclined substrate and the influence of various parameters on the behavior of traveling wave solutions are investigated in \cite{Shearer1, Shearer2, Shearer3}. We also refer to \cite{BarrettGarcke, GruenLenz} for finite volume methods and numerical simulations for thin liquid films with insoluble surfactant. We shall point out that research has also been dedicated to the mathematical analysis for the free boundary problem for two-phase flows with soluble surfactant, see \cite{Bothe1, Bothe2}. Finally, we refer to the forthcoming paper \cite{EHLW} where existence of global weak solutions will be shown for the case of insoluble surfactant under moderate assumptions on the surface tension coefficient.

The outline of this article is as follows. In Section \ref{sec2} we provide a derivation of the governing equations \eqref{1}-\eqref{4}. In Section \ref{sec3} we begin our mathematical analysis of these equations. We establish a local existence result for smooth solutions and provide an energy function which shows that the only steady states are constants. Based on the principle of linearized stability we then show in Section \ref{sec4} that steady states with small surfactant concentrations are locally asymptotically stable.

\section{The Physical Model}\label{sec2}

The derivation of the governing equations for the dynamics of the liquid thin films with surfactant using lubrication theory can be found in many papers for various situations, see, e.g., \cite{GrotbergJensen92, GrotbergJensen93, LinHwang00, LinHwang02, Matar02, WarnerCraster} to name a few. The aim of this section is to provide a brief sketch of the derivation of the model equations \eqref{1}-\eqref{4} for a soluble surfactant including gravity that we have not found in the literature in this particular form. For this, we follow \cite{GrotbergJensen92, GrotbergJensen93} to which we also refer for further details. 

We shall consider a thin film of a viscous incompressible Newtonian fluid lying on a horizontal plane. Initially a monolayer of soluble surfactant is deposited on the surface of the film. The ratio $\ve$ of the undisturbed height compared to the initial length of the film is taken to be small so that lubrication theory may be used. In the following, all variables are non-dimensionalized. Let $x$ and $z$ denote the horizontal and vertical coordinate, respectively. Let $(u(t,x,z), w(t,x,z))$ denote the corresponding velocity field of the fluid. The free surface of the fluid layer is at $z=h(t,x)$ while $\Gamma(t,x)$ and $C(t,x,z)$ denote the monolayer concentration and the bulk concentration of the surfactant, respectively, see Figure~\ref{fig:tf}.
\begin{figure}[h!]
\begin{minipage}[t]{1\linewidth}
\centering\includegraphics[width=10cm]{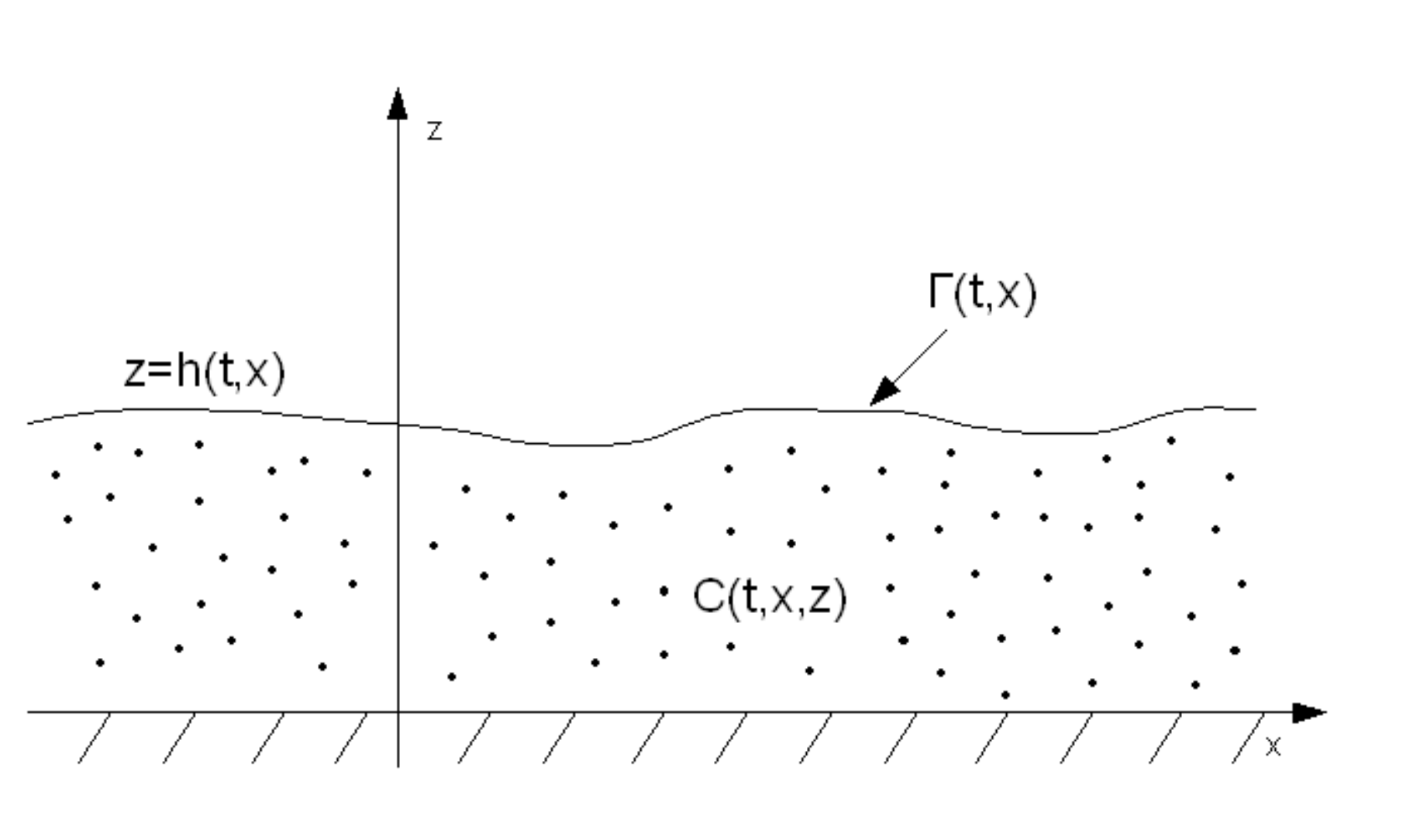}
\caption{\small Schematic representation of a thin film with soluble surfactant.}\label{fig:tf}
 \end{minipage} \hfill
\end{figure}
Conservation of mass of the fluid reads
\beqn\label{5}
\partial_xu(t,x,z)+\partial_z w(t,x,z)=0\ ,
\eeqn
while the equations of momentum conservation are given by
\begin{align}
&-\partial_x p(t,x,z)+\partial_z^2 u(t,x,z)=0\ ,\label{6}\\
&-\partial_z p(t,x,z)-G =0\ ,\label{7}
\end{align}
where $p$ is the pressure in the liquid and $G>0$ represents a gravitational force directed vertically downwards. In \eqref{6}, \eqref{7} we neglect van der Waals forces and capillarity effects, the latter assumption being valid if $\sigma_0/S= O(1)$, where $S$ is the spreading coefficient and $\sigma_0$ is the surface tension of the uncontaminated interface, see \cite{GrotbergJensen93}. At the bottom plane $z=0$, no-slip boundary conditions are imposed:
\beqn\label{8}
u(t,x,0)=w(t,x,0)=0\ ,
\eeqn
whereas at the interface $z=h(t,x)$ the kinematic boundary condition 
\beqn\label{9}
\partial_t h(t,x)+u\big(t,x,h(t,x)\big)\partial_x h(t,x)=w\big(t,x,h(t,x)\big)
\eeqn
is prescribed meaning that at the free surface the speed of the interface balances the normal component of the liquid velocity. Assuming zero pressure above the film and again neglecting effects of capillarity, the interface condition at $z=h(t,x)$ is complemented with a normal and tangential stress condition reading
\beqn\label{10}
p\big(t,x,h(t,x)\big)=0
\eeqn
and
\beqn\label{11}
\partial_z u\big(t,x,h(t,x)\big)=\partial_x \sigma\big(\Gamma(t,x)\big)\ .
\eeqn
Here $\sigma(\Gamma)$ denotes the surface tension of the liquid related to the local surface surfactant concentration. Examples of constitutive relations (in dimensionless form) for the surface tension include $\sigma(\Gamma)=1-\Gamma$ or 
$$
\sigma(\Gamma)= (\alpha+1)\ \left( 1 - \left[ ((\alpha+1)/\alpha)^{1/3} - 1 \right]\ \Gamma \right)^{-3}\,,
$$
see \cite{deWitGallez,GrotbergJensen92,GrotbergJensen93,WarnerCraster}. The hydrodynamics is supplemented by mass transfer of the soluble surfactant which is supposed to be deposited initially at the surface of the interface as a monolayer. The dynamics of the surfactant surface concentration $\Gamma$ at the free surface $z=h(t,x)$ is governed by the advection-transport equation
\beqn\label{12}
\partial_t\Gamma(t,x)+\partial_x\big(u\big(t,x,h(t,x)\big)\,\Gamma(t,x)\big)=D\,\partial_x^2\Gamma(t,x)+K\,\big(C_s(t,x)-\Gamma(t,x)\big)
\eeqn
with $D>0$ being a non-dimensional surface diffusion coefficient (independent of the surfactant concentration). On the right-hand side of \eqref{12}, 
$K\,\big(C_s(t,x)-\Gamma(t,x)\big)$ with
\beqn\label{13}
C_s(t,x)=C\big(t,x,h(t,x)\big)
\eeqn
is the (scaled) diffusion controlled sorptive flux of surfactant between the surface and the region of the bulk immediately beneath the surface, where $C$ denotes the surfactant concentration in the bulk. The constant $K$ is the ratio of the time scale of the flow to the time scale of desorption. The transport equation for the bulk surfactant concentration reads
\beqn\label{14}
\partial_t C(t,x,z)+\partial_x(u\,C)(t,x,z)+\partial_z(w\,C)(t,x,z)=\delta\,\partial_x^2 C(t,x,z)+\delta_V\,\partial_z^2 C(t,x,z)\ ,
\eeqn
the positive numbers $\delta$ and $\delta_V$ in \eqref{14} representing the non-dimensional horizontal and vertical bulk diffusivities independent of the surfactant concentration. Equation \eqref{14} is supplemented with the boundary conditions
\beqn\label{15}
\delta_V\,\partial_z C\big(t,x,h(t,x)\big)-\delta\, \partial_x h(t,x)\,\partial_x C\big(t,x,h(t,x)\big)=-\beta\, K\, \big(C_s(t,x)-\Gamma(t,x)\big)
\eeqn
at the surface $z=h(t,x)$ and
\beqn\label{16}
\partial_z C(t,x,0)=0
\eeqn
at the bottom $z=0$. The constant $\beta>0$ appearing in \eqref{15} is proportional to the ratio of the rate of adsorption to the rate of desorption, and indicates the degree of solubility. The limit $\beta\rightarrow 0$ corresponds to the case of a surfactant with high substrate solubility only weakly adsorbing at the free surface, while in the limit $\beta\rightarrow\infty$ the surfactant has high surface adsorption and is only weakly soluble in the bulk. On the vertical boundaries $x=0$ and $x=L$ we assume that there is no flux of surfactant, i.e. 
\beqn\label{17}
\partial_x\Gamma(t,x)=\partial_x C(t,x,z)=0\ ,\quad x=0,\, L\ .
\eeqn

Next, we shall compute $u$, $w$, and $p$ from equations \eqref{6}-\eqref{17} and then use cross-sectional averaging for $C$ in order to derive the one-dimensional model equations \eqref{1}-\eqref{3}. 

Notice first that \eqref{7} and \eqref{10} yield
$$
p(t,x,z)=G\,\big(h(t,x)-z\big)\ .
$$
Together with \eqref{6} this then implies that
$$
\partial_z u(t,x,h(t,x))-\partial_zu(t,x,z)=G\,\partial_x h(t,x)\,\big(h(t,x)-z\big)\ ,
$$
and thus, invoking \eqref{11}, we obtain that
$$
\partial_z u(t,x,z)=\partial_x\sigma(\Gamma(t,x))-G\,\partial_xh(t,x)\,\big(h(t,x)-z\big)\ ,
$$
from which, in turn, we deduce by \eqref{8} that
\beqn\label{18}
u(t,x,z)=z\,\partial_x\sigma(\Gamma(t,x))-G\,\partial_x h(t,x)\,\left(z\,h(t,x)-\dfrac{z^2}{2}\right)\ .
\eeqn
Integrating \eqref{5} and using \eqref{8} we see that
\beqn\label{18b}
\int_0^{h(t,x)}\partial_x u(t,x,z)\,\rd z = -w(t,x,h(t,x))\ ,
\eeqn
whence, by \eqref{9},
\beqn\label{18c}
\partial_t h(t,x)+\partial_x\left(\int_0^{h(t,x)} u(t,x,z)\, \rd z\right)=0\ .
\eeqn
Recalling \eqref{18} we derive the evolution equation \eqref{1} for the film thickness (for $t>0$ and $0<x<L$):
\beqn\label{19}
\partial_t h-\partial_x\left(\frac{G}{3}\,h^3\,\partial_x h-\frac{1}{2}\, h^2\,\partial_x\sigma(\Gamma)\right)=0\ .
\eeqn
Inserting next \eqref{18} evaluated at $z=h(t,x)$ into \eqref{12}, the transport equation for the surface surfactant concentration becomes (for $t>0$ and $0<x<L$)
\beqn\label{20}
\partial_t\Gamma-\partial_x\left(\frac{G}{2}\,h^2\,\Gamma\,\partial_xh-h\,\Gamma\,\partial_x\sigma(\Gamma)\right)=D\,\partial_x^2\Gamma+K\,(C_s-\Gamma)\ .
\eeqn
We now introduce the cross-sectional average of $C$ defined by
$$
C_0(t,x) := \frac{1}{h(t,x)}\ \int_0^{h(t,x)} C(t,x,z)\, \rd z\ .
$$
It follows from \eqref{14} and \eqref{16} that
\begin{eqnarray*}
\partial_t (hC_0)(t,x) & = & \int_0^{h(t,x)} \partial_t C(t,x,z)\, \rd z + C\big(t,x,h(t,x)\big)\ \partial_t h(t,x) \\
& = & - \int_0^{h(t,x)} \partial_x (uC)(t,x,z)\, \rd z + C\big(t,x,h(t,x)\big)\ \big( - w(t,x,h(t,x)) + \partial_t h(t,x) \big) \\
&  &+\ \delta\ \int_0^{h(t,x)} \partial_x^2 C(t,x,z)\, \rd z + \delta_V\ \partial_z C\big(t,x,h(t,x)\big)\,.
\end{eqnarray*}
Owing to \eqref{9} and \eqref{15}, we obtain
\begin{eqnarray*}
\partial_t (hC_0)(t,x) & = & - \int_0^{h(t,x)} \partial_x (uC)(t,x,z)\, \rd z - (uC)\big(t,x,h(t,x)\big)\ \partial_x h(t,x) \\
&  &+\ \delta\ \int_0^{h(t,x)} \partial_x^2 C(t,x,z)\, \rd z + \delta\ \partial_x C\big(t,x,h(t,x)\big)\ \partial_x h(t,x) \\
&  &-\ \beta\, K\, \big(C_s(t,x)-\Gamma(t,x)\big) \\
& = & - \partial_x \left( \int_0^{h(t,x)} \big[ (uC)(t,x,z) - \delta\ \partial_x C(t,x,z) \big]\, \rd z \right) \\
&  &-\ \beta\, K\, \big(C_s(t,x)-\Gamma(t,x)\big)\,.
\end{eqnarray*}
At this stage, we may assume that, within the lubrication approximation, $C(t,x,z)$, $\partial_x C(t,x,z)$, and $C_s(t,x)$ are well approximated by $C_0(t,x)$, $\partial_x C_0(t,x)$, and $C_0(t,x)$, respectively, and we deduce from \eqref{18} that, for $t>0$ and $0<x<L$,
\beqn\label{23}
\partial_t(h\, C_0)-\partial_x\left(\frac{G}{3}\,h^3\, C_0\,\partial_x h-\frac{1}{2}\,h^2\,C_0\,\partial_x\sigma(\Gamma)+\delta\, h\,\partial_x C_0\right)=-\beta \,K\, (C_0-\Gamma)\ .
\eeqn
Therefore, gathering \eqref{19}, \eqref{20}, and \eqref{23} we obtain the model equations \eqref{1}-\eqref{4}.
Notice that
\beqn\label{24}
\frac{\rd}{\rd t}\int_0^L h(t,x)\,\rd x=0\ ,\qquad \frac{\rd}{\rd t}\int_0^L \left(\Gamma(t,x)+\frac{1}{\beta}\,h(t,x)\, C_0(t,x)\right)\,\rd x=0\ ,
\eeqn
that is, the mass of the fluid and the mass of the surfactant are conserved during time evolution, the factor $1/\beta$ accounting for the degree of solubility of the surfactant in the fluid. \\

For the mathematical analysis of \eqref{1}-\eqref{4} it is convenient to introduce 
a new variable \mbox{$m:=hC_0/\beta$} replacing $C_0$. Then the system \eqref{1}-\eqref{4} becomes (for non-vanishing $h$):
\begin{align}
&\partial_t h-\partial_x\left(\frac{G}{3}\,h^3\,\partial_x h-\frac{1}{2}\, h^2\,\partial_x\sigma(\Gamma)\right)=0\ ,\label{1A}\\
&\partial_tm-\partial_x\left(\frac{G}{3}\,h^2\, m\,\partial_x h-\frac{1}{2}\,h\, m\,\partial_x\sigma(\Gamma)+\delta\, \partial_x m-\delta\,\frac{m}{h}\,\partial_x h\right)=- K\, \left(\frac{\beta\,m}{h}-\Gamma\right)\ ,\label{2A}\\
&\partial_t\Gamma-\partial_x\left(\frac{G}{2}\,h^2\,\Gamma\,\partial_xh-h\,\Gamma\,\partial_x\sigma(\Gamma)+D\,\partial_x\Gamma\right)=K\, \left(\frac{\beta\,m}{h}-\Gamma\right)\label{3A}
\end{align}
for $x\in (0,L)$ and $t>0$ subject to the initial conditions at $t=0$,
\beqn\label{4A}
h(0,\cdot)=h^0\ ,\quad m(0,\cdot)=m^0\ ,\quad \Gamma(0,\cdot)=\Gamma^0\ , \quad x\in (0,L)\ ,
\eeqn
and no-slip boundary conditions at $x=0$ and $x=L$,
\beqn\label{4AA}
\partial_xh=\partial_x m=\partial_x\Gamma=0\ ,\quad x=0, L\ .
\eeqn
In the following we shall focus on this version of the system \eqref{1}-\eqref{4}.

\section{Local Well-Posedness and an Energy Functional}\label{sec3}

First we state an existence result (locally in time) for the system \eqref{1A}-\eqref{4AA}. For this we define 
$$
O_\alpha:=H_N^{2\alpha}\big((0,L),\R^3\big)\cap C\big([0,L],(0,\infty)^3\big)\ ,\quad \alpha\in  [0,1]\ ,
$$
where
$$
H_N^{2\alpha}:=H_N^{2\alpha}\big((0,L),\R^3\big):=\left\{
\begin{array}{lcl}
\big\{u\in H^{2\alpha}\big((0,L),\R^3\big)\,;\, \partial_xu=0\ \text{for}\ x=0,L\big\} & \mbox{ if } & \alpha>3/4\,,\\ 
& & \\
H^{2\alpha}\big((0,L),\R^3\big) & \mbox{ if } & \alpha\in [0,3/4]\,,
\end{array}
\right.
$$
with $H^{2\alpha}$ being the usual Sobolev spaces. Let $\alpha\in (3/4,1]$. Then $O_\alpha$ is open in $H_N^{2\alpha}$ and a subset of $C^1\big([0,L],(0,\infty)^3\big)$.
Given $u:=(h,m,\Gamma)$, we introduce the matrices
\begin{align*}
&a(u):=a(h,m,\Gamma):=\left(
\begin{matrix}
\displaystyle{\frac{G}{3}\,h^3}&0&\displaystyle{-\frac{h^2}{2}\,\sigma'(\Gamma)}\\ 
& & \\
\displaystyle{\frac{G}{3}\,h^2\, m-\delta\,\frac{m}{h}}& \delta&\displaystyle{-\frac{1}{2}\,h\,m\,\sigma'(\Gamma)}\\ 
& & \\
\displaystyle{\frac{G}{2}\, h^2\,\Gamma} & 0&\displaystyle{D-h\,\Gamma\,\sigma'(\Gamma)}
\end{matrix}\right)\ ,\\
& \\
&b(u):=b(h,m,\Gamma):=\left(
\begin{matrix}
0&0&0\\
& & \\
0&\displaystyle{\frac{K\,\beta}{h}}&-K\\ 
& & \\
0&\displaystyle{-\frac{K\,\beta}{h}}&K
\end{matrix}\right)\ ,
\end{align*}
and we define an operator $A$ as
\beqn\label{400}
A(u)w:=-\partial_x\big(a(u)\partial_x w\big)+b(u)w\ ,\qquad w\in H_N^{2}\ ,\quad u\in O_\alpha\ .
\eeqn
Writing $u^0:=(h^0,m^0,\Gamma^0)$ equations \eqref{1A}-\eqref{4AA} may be recast as a quasi-linear equation in the space \mbox{$L_2:=L_2\big((0,L),\R^3\big)$} of the form
\beqn\label{30}
\dot{u} +A(u)\,u=0\ ,\quad t>0\ ,\qquad u(0)=u^0\ .
\eeqn
If $\sigma$ is a non-increasing function, i.e. $\sigma'\le 0$, it readily follows that the matrix $a(\xi_1,\xi_2,\xi_3)$ for $\xi_j\in (0,\infty)$ has only positive eigenvalues. Thus, since $b(u)$ defines (for $u\in O_\alpha$ fixed) a bounded operator on $L_2$, we infer from \cite[Ex.~4.3.e), Thm.~4.1]{AmannTeubner} that
\beqn\label{40}
A\in C^{1-}\big(O_\alpha,\mathcal{H}(H_N^{2},L_2)\big)\ ,
\eeqn
that is, $A$ is a locally Lipschitz mapping from $O_\alpha$ into the set of negative generators of strongly continuous analytic semigroups on $L_2$ with domain $H_N^{2}$. Therefore, \cite[Thm.~12.1]{AmannTeubner} warrants the following local existence result concerning equation \eqref{30}:

\begin{thm}\label{T1}
Let $\sigma\in C^2(\R)$ be non-increasing, and let $\alpha\in (3/4,1]$. Then, given any initial condition \mbox{$u^0=(h^0,m^0,\Gamma^0)\in O_\alpha$}, the problem \eqref{30} (and hence \eqref{1A}-\eqref{4AA}) admits a unique (strictly) positive strong solution
$$
u=(h,m,\Gamma)\in C\big([0,T),O_\alpha\big)\cap C^1\big((0,T),L_2\big)\cap C\big((0,T),H_N^2\big)
$$
with maximal time of existence $T\in (0,\infty]$.
\end{thm}

Notice that, thanks to the positivity of $h$, $C_0:=\beta m/h$ is well-defined having the same regularity properties as $m$ and $(h,C_0,\Gamma)$ solves \eqref{1}-\eqref{4}. Notice further that the positivity of solutions is built into the set $O_\alpha$.\\

Next we shall derive an energy functional for the system \eqref{1A}-\eqref{4AA}.

\begin{prop}\label{P1.5}
Consider a non-increasing function $\sigma\in C^2(\R)$, $\alpha\in (3/4,1]$, and an initial condition \mbox{$u^0=(h^0,m^0,\Gamma^0)\in O_\alpha$}. The corresponding solution $u=(h,m,\Gamma)$ to \eqref{1A}-\eqref{4AA} given by Theorem~\ref{T1} satisfies
\bqnn
\begin{split}
\dfrac{\rd}{\rd t}\int_0^L & \left(\phi(\Gamma)+\frac{1}{\beta}\,h\, \phi\left(\frac{\beta\, m}{h}\right)+\frac{G}{2}\, h^2\right)\, \rd x\\
&=\ -D\int_0^L \phi''(\Gamma)\,\vert\partial_x\Gamma \vert^2\,\rd x -\frac{\delta}{\beta}\int_0^L \phi''\left(\frac{m}{h}\right)\,h\,\left\vert \partial_x\left(\frac{\beta\, m}{h}\right)\right\vert^2\,\rd x\\
&\quad\ -\frac{1}{4}\int_0^L h\,\vert\partial_x\sigma(\Gamma) \vert^2\,\rd x-\int_0^L\left(\frac{G}{\sqrt{3}}\, h^{3/2}\, \partial_x h-\frac{\sqrt{3}}{2}\, h^{1/2}\,\partial_x\sigma(\Gamma)\right)^2\,\rd x\\
&\quad\ -K\int_0^L\left(\phi'(\Gamma)-\phi'\left(\frac{\beta\, m}{h}\right)\right)\,\left( \Gamma-\frac{\beta\, m}{h}\right)\, \rd x\ ,
\end{split}
\eqnn
for $t\in [0,T)$, the function $\phi$ being such that 
\beqn\label{31}
\phi''(r)\,r=-\sigma'(r)\ge 0\ ,\quad r> 0\ .
\eeqn
\end{prop}

Observe that each term on the right-hand side is non-positive, in particular due to the monotonicity \eqref{31} of $\phi'$.

\begin{proof}
We deduce from \eqref{1A}-\eqref{4AA} that, for $t\in [0,T)$,
\bqnn
\begin{split}
\dfrac{\rd}{\rd t}\int_0^L & \left(\phi(\Gamma)+\frac{1}{\beta}\,h\, \phi\left(\frac{\beta\, m}{h}\right)\right)\, \rd x\\
&=\int_0^L\left\{\phi'(\Gamma)\,\partial_t\Gamma+\phi'\left(\frac{\beta\, m}{h}\right)\,\partial_t m+\frac{1}{\beta}\,\left(\phi\left(\frac{\beta\, m}{h}\right)-\frac{\beta\, m}{h}\phi'\left(\frac{\beta\, m}{h}\right)\right)\,\partial_th\right\}\,\rd x\\
&=\int_0^L\phi''(\Gamma)\,\partial_x\Gamma\,\left(-\frac{G}{2}\, h^2\,\Gamma\,\partial_x h+h\,\Gamma\,\partial_x\sigma(\Gamma)-D\,\partial_x\Gamma\right)\,\rd x\\
&\quad +K\int_0^L\phi'(\Gamma)\,\left( \frac{\beta\,m}{h}-\Gamma\right)\, \rd x\\
&\quad +\int_0^L\phi''\left(\frac{\beta\, m}{h}\right)\,\partial_x\left(\frac{\beta\, m}{h}\right)\, 
\left\{ \left(-\frac{G}{3}\,h^2\,m+\delta\,\frac{m}{h}\right)\, \partial_x h-\delta\,\partial_x m+\frac {1}{2}\,m\,h\,\partial_x\sigma(\Gamma)\right\}\,\rd x\\
&\quad - K\int_0^L\phi'\left(\frac{\beta\, m}{h}\right)\,\left( \frac{\beta\, m}{h}-\Gamma\right)\, \rd x\\
&\quad-\frac{1}{\beta}\int_0^L \phi''\left(\frac{\beta\, m}{h}\right)\,\frac{\beta\, m}{h}\,\partial_x \left(\frac{\beta\, m}{h}\right)\,\left\{-\frac{G}{3}\,h^3\,\partial_x h+\frac{1}{2}\,h^2\,\partial_x\sigma(\Gamma)\right\}\,\rd x\ .
\end{split}
\eqnn
Recall that \eqref{31} ensures 
$$
\phi''(\Gamma)\Gamma\partial_x\Gamma=-\partial_x\sigma(\Gamma) \;\;\mbox{ and }\;\; h\partial_x\left( \frac{\beta m}{h} \right)=\beta\partial_x m-\frac{\beta m}{h}\partial_x h\ .
$$ 
Hence, noticing that the last integral of the right-hand side of the above equality cancels with the first and the last term of the third integral, we have
\bqnn
\begin{split}
\dfrac{\rd}{\rd t}\int_0^L & \left(\phi(\Gamma)+\frac{1}{\beta}\,h\, \phi\left(\frac{\beta\, m}{h}\right)\right)\, \rd x\\
&=-\int_0^L h\,\vert\partial_x\sigma(\Gamma) \vert^2\,\rd x-D\int_0^L \phi''(\Gamma)\,\vert\partial_x\Gamma \vert^2\,\rd x  -\frac{\delta}{\beta}\int_0^L \phi''\left(\frac{\beta\, m}{h}\right)\,h\,\left\vert \partial_x\left(\frac{\beta\, m}{h}\right)\right\vert^2\,\rd x\\
&\quad -K\int_0^L\left(\phi'(\Gamma)-\phi'\left(\frac{\beta\, m}{h}\right)\right)\,\left( \Gamma-\frac{\beta\, m}{h}\right)\, \rd x +\frac{G}{2}\int_0^L h^2\,\partial_x h\,\partial_x\sigma(\Gamma)\, \rd x\ .
\end{split}
\eqnn
From equation \eqref{1A} it follows that
\bqnn
\dfrac{\rd}{\rd t} \frac{G}{2}\int_0^L h^2\,\rd x=-\frac{G^2}{3}\int_0^L h^3\,\vert\partial_x h\vert^2\,\rd x +\frac{G}{2}\int_0^L h^2\, \partial_x h\,\partial_x\sigma(\Gamma)\,\rd x\ ,
\eqnn
which, added to the previous equality and after completing the square, yields the stated energy equality.   
\end{proof}

\medskip

An interesting consequence of Proposition~\ref{P1.5} is that it provides a complete description of the (smooth positive) stationary solutions to \eqref{1A}-\eqref{4AA} when $\sigma$ is decreasing. Indeed, observe that, if $u$ is a (smooth positive) stationary solution to \eqref{1A}-\eqref{4AA}, the left-hand side of the energy equality vanishes and so does each term of the right-hand side since they are all nonnegative. It is then straightforward to establish the following result:

\begin{cor}\label{C2}
Suppose that $\sigma\in C^2(\R)$ is strictly decreasing. Then the only (smooth positive) steady states to \eqref{1A}-\eqref{4AA} are of the form $(h_*, m_*,\Gamma_*)$ with constants $h_*>0$ and $m_*=h_*\Gamma_*/\beta >0$.
\end{cor}

The next section is dedicated to the local asymptotic stability of the steady states $(h_*, m_*,\Gamma_*)$ for small values of the surfactant concentrations, i.e. for small values of $m_*$ and $\Gamma_*$. Let us mention at this point that some constraints have to be satisfied by the initial data because of the conservation of the mass of the fluid and the mass of the surfactant \eqref{24}.

\section{Asymptotic Stability}\label{sec4}

To prove stability of steady states we use the same notation as in the previous section and write equations \eqref{1A}-\eqref{4AA} in the form \eqref{30} with $u=(h,m,\Gamma)$ and $u^0=(h^0,m^0,\Gamma^0)$. 

Let $h_*>0$ be a fixed constant and, given a sufficiently small number $\eta_*>0$, set
$$
m_*:=\frac{h_*}{\beta+h_*}\, \eta_*\ ,\qquad \Gamma_*:=\frac{\beta}{\beta+h_*}\, \eta_*\ \ .
$$
Notice that $\beta m_*=h_*\Gamma_*$ and thus, $u_*:=(h_*,m_*,\Gamma_*)$ is an equilibrium of \eqref{30}.
Motivated by \eqref{24} we then introduce a projection $P\in\mathcal{L}(L_2)\cap\mathcal{L}(H_N^2)$, i.e. $P^2=P$, by setting
$$
Pu:=\left(\begin{matrix} 
\displaystyle{h-\langle h\rangle} \\ 
 \\
\displaystyle{m-\frac{h_*}{h_*+\beta}\,\langle m+\Gamma\rangle}\\ 
 \\
\displaystyle{\Gamma-\frac{\beta}{h_*+\beta}\,\langle m+\Gamma\rangle}
\end{matrix}\right)\ ,\quad u=(h,m,\Gamma)\in L_2\, ,
$$
where
$$
\langle f\rangle:=\frac{1}{L}\int_0^L f(x)\,\rd x
$$
denotes the mean value of a given function $f\in L_2((0,L))$. Then both spaces $L_2$ and $H_N^2$ decompose into
$$
L_2=PL_2\oplus (1-P)L_2\ ,\qquad H_N^2=PH_N^2\oplus (1-P)H_N^2\ .
$$
Also note that
\beqn\label{51}
(1-P)\, A(v)\, w=0\ ,\quad (1-P)\,b(v)\, w=0\ , \qquad w\in H_N^2\ ,\quad v\in O_1\ ,
\eeqn
that is, $A(v)w\in PL_2$ and $b(v)w\in PL_2$ for any $v\in O_1$ and $w\in H_N^2$. Thus, if $u$ is any strong solution to \eqref{30} on $[0,\infty)$ with initial condition $u^0=(h^0,m^0,\Gamma^0)\in H_N^2$ satisfying $\langle h^0\rangle=h_*$ and $\langle m^0+\Gamma^0\rangle=\eta_*$, then necessarily
$$
(1-P)u(t)=(1-P)u^0=u_*\ ,\quad t\ge 0\ .
$$ 
Therefore, we have
$$
u(t)=v(t) + u_*\,\in PH_N^2\oplus (1-P)H_N^2\ ,\quad t>0\ ,
$$ 
and $v:=Pu$ solves
\begin{align}
\dot{v}+A_*v&=\big(A_*-A(u_*+v)\vert_{PH_N^2}\big)\, v-b(u_*+v)u_*=:F(v)\ ,\quad t>0\ ,\label{52a}\\
v(0)&=Pu^0=u^0-u_*=:v^0\ ,\label{52b}
\end{align}
with
$$
A_* w:=A(u_*)\vert_{PH_N^2}\, w +B_*\vert_{PH_N^2}\, w\ ,\quad w\in PH_N^2\ .
$$
Here, the operator $B_*$, given by 
\beqn\label{53}
B_*\, w=\left(\begin{matrix} 0\\ 
 \\
\displaystyle{-\frac{K\,\beta\, m_*}{h_*^2} h}\\
 \\
\displaystyle{\frac{K\,\beta\, m_*}{h_*^2} h}
\end{matrix}\right)\ ,\quad w=(h,m,\Gamma)\in H^1\ ,
\eeqn
is the Fr\'{e}chet derivative of $[z\mapsto b(z)u_*]: O_{1/2}\rightarrow L_2$ at $z=u_*$.

\begin{lem}\label{L1}
The operator $A_*$, considered as an unbounded operator in $PL_2$ with domain $PH_N^2$, belongs to $\mathcal{H}(PH_N^2,PL_2)$, that is, $-A_*$ is the generator of a strongly continuous analytic semigroup on $PL_2$.
\end{lem}

\begin{proof} Noticing that $A(u_*)\vert_{(1-P)H_N^2}=0$ and recalling \eqref{40} and \eqref{51}, we may interpret $A(u_*)$ as a matrix operator
$$
\left(\begin{matrix} A(u_*)\vert_{PH_N^2} & 0\\ 0&0\end{matrix}\right)\in\mathcal{H}\big(PH_N^2\oplus (1-P)H_N^2,PL_2\oplus (1-P)L_2\big)\ .
$$
We then conclude from \cite[I.Cor.~1.6.3]{LQPP} that $A(u_*)\vert_{PH_N^2}\in \mathcal{H}(PH_N^2,PL_2)$. Now $B_*\vert_{PH_N^2}\in\mathcal{L}(H_N^2,PH_N^2)$ can be considered as a perturbation, see \cite[I.Thm.~1.3.1]{LQPP}.
\end{proof}

Next we show that $-A_*$ has negative type for small values of $\eta_*$. 

\begin{lem}\label{L2}
There are numbers $\eps:=\eps(h_*)>0$ and \mbox{$\omega_0:=\omega_0(h_*)>0$} such that the spectrum $\sigma(-A_*)$ of $-A_*$ is contained in the half-plane $[\mathrm{Re}\,\lambda\le -\omega_0]$ provided that $0<\eta_*<\eps$.
\end{lem}

\begin{proof}
Let $w^0=(h^0,m^0,\Gamma^0)\in PL_2$ be arbitrary and let $w(t):=e^{-tA_*}w^0$, $t\ge 0$, be the unique solution in $PL_2$ to
the linear equation
$$
\dot{w}+A_*w=0\ ,\quad t>0\ ,\qquad w(0)=w^0\ .
$$
Writing $w=(h,m,\Gamma)$ we have
\beqn\label{59}
\langle h(t)\rangle =0\ ,\quad \langle m(t)+\Gamma(t)\rangle=0\ ,\qquad t\ge 0\ ,
\eeqn
and, by definition of $A_*$ and \eqref{53},
\begin{align*}
&\partial_t h-\partial_x\left(\frac{G}{3}\,h_*^3\,\partial_x h-\frac{1}{2}\, h_*^2\,\sigma'(\Gamma_*)\,\partial_x\Gamma\right)=0\ ,\\
&\partial_tm-\partial_x\left(\frac{G}{3}\,h_*^2\, m_*\,\partial_x h-\frac{1}{2}\,h_*\, m_*\,\sigma'(\Gamma_*)\,\partial_x\Gamma+\delta\, \partial_x m-\delta\,\frac{m_*}{h_*}\,\partial_x h\right)\\
&\qquad\qquad\qquad +\frac{K\,\beta}{h_*}\,m-
K\, \Gamma -K\,\beta\,\frac{m_*}{h_*^2}\, h=0\ ,\\
&\partial_t\Gamma-\partial_x\left(\frac{G}{2}\,h_*^2\,\Gamma_*\,\partial_x h+\big(D - h_*\,\Gamma_*\,\sigma'(\Gamma_*)\,\big)\,\partial_x\Gamma\right) -\frac{K\,\beta}{h_*}\,m+ K\, \Gamma +K\,\beta\,\frac{m_*}{h_*^2}\, h=0\ ,
\end{align*}
for $x\in (0,L)$ and $t>0$ subject to
\bqnn
\partial_xh=\partial_x m=\partial_x\Gamma=0\ ,\quad x=0, L\ ,\quad t>0\ .
\eqnn
Given $q>0$ we define the symmetric matrix
$$
b_q(h_*,m_*,\Gamma_*):=\left(\begin{matrix} 
\displaystyle{\frac{q\,G}{3}\,h_*^3} & \displaystyle{\frac{\beta\, G}{6}\,h_*^2\,m_* - \frac{\delta\,\beta}{2}\, \frac{m_*}{h_*}} & \displaystyle{\frac{G}{4}\,h_*^3\,\Gamma_*-\frac{q}{4}\,h_*^2\,\sigma'( \Gamma_*)}\\  
 & & \\
\displaystyle{\frac{\beta\, G}{6}\, h_*^2\, m_* - \frac{\delta\,\beta}{2}\,\frac{m_*}{h_*}} & \displaystyle{\delta\,\beta} &\displaystyle{-\frac{\beta}{4}\, m_*\, h_*\, \sigma'(\Gamma_*)}\\ 
 & & \\
\displaystyle{\frac{G}{4}\, h_*^3\,\Gamma_*-\frac{q}{4}\,h_*^2\,\sigma'( \Gamma_*)} & \displaystyle{-\frac{\beta}{4}\, m_*\, h_*\, \sigma'(\Gamma_*)} & \displaystyle{D\, h_*-h_*^2\,\Gamma_*\, \sigma'(\Gamma_*)} 
\end{matrix}\right)\ .
$$
Then, multiplying the equations satisfied by $h$, $m$, and $\Gamma$ by $q h$, $\beta m$, and $h_*\Gamma$, respectively, we derive the equality
\beqn\label{60}
\begin{split}
\frac{1}{2}\frac{\rd }{\rd t}\left(q\,\|h\|_2^2+\beta \,\|m\|_2^2+h_*\,\|\Gamma\|_2^2\right)& +\left(b_q(h_*,m_*,\Gamma_*)\partial_x w \Big\vert \partial_x w \right)_2\\
&+\frac{K}{h_*}\,\|\beta\, m-h_*\Gamma\|_2^2 +\frac{K\,\beta\, m_*}{h_*^2}\big(h\vert h_*\,\Gamma-\beta m\big)_2=0\ ,
\end{split}
\eeqn
where $\|\cdot\|_2$ and $(\cdot\vert\cdot)_2$ denote the norm and inner product in $L_2$.
Next, since
$$
b_q(h_*,0,0)=\left(\begin{matrix} 
\displaystyle{\frac{q\,G}{3}\,h_*^3} & 0 & \displaystyle{- \frac{q}{4}\,h_*^2\,\sigma'(0)}\\  
 & & \\
0 & \delta\,\beta &0 \\ 
 & & \\
\displaystyle{-\frac{q}{4}\,h_*^2\,\sigma'(0)} & 0& D\, h_* 
\end{matrix}\right)\ ,
$$
we have
$$
\mathrm{det}\big(b_q(h_*,0,0)-\lambda\big)=(\delta\,\beta-\lambda)\left[\lambda^2-\left(\frac{q}{3}\,G\,h_*^3+D\,h_*\right)\,\lambda+\frac{q}{3}\,G\,D\,h_*^4-\frac{1}{16}\,q^2\,h_*^4\,\sigma'(0)^2\right]\ ,
$$
so that all eigenvalues of $b_q(h_*,0,0)$ are positive provided that
$$
0<q<\frac{16\,G\,D}{3\,\sigma'(0)^2}\ .
$$
Thus, for $q$ satisfying this condition, there is some $\eps:=\eps(h_*)>0$ such that the matrix $b_q(h_*,m_*,\Gamma_*)$ is positive definite for $0<\eta_*<\eps$ by continuity and the definitions of $m_*$ and $\Gamma_*$. From \eqref{60} we then derive that
\beqn\label{61}
\begin{split}
\frac{1}{2}\frac{\rd }{\rd t}\left(q\,\|h\|_2^2+\beta\, \|m\|_2^2+h_*\,\|\Gamma\|_2^2\right)& +\mu\left(\|\partial_xh\|_2^2+ \|\partial_xm\|_2^2+\|\partial_x\Gamma\|_2^2\right) +\frac{K}{h_*}\,\|\beta\, m-h_*\Gamma\|_2^2\\
& \le \frac{K\,\beta\, m_*}{2\, h_*^2}\big(\|h\|_2^2+\|\beta\, m-h_*\Gamma\|_2^2\big)\
\end{split}
\eeqn
for some number $\mu:=\mu(h_*)>0$. Recalling \eqref{59} we may apply Poincar\'e's inequality to deduce that
\beqn\label{63}
\|\partial_x m\|_2^2+\|\partial_x\Gamma\|_2^2\ge\frac{1}{2}\,\|\partial_xm+\partial_x\Gamma\|_2^2\ge c\,\|m+\Gamma\|_2^2\ \quad\text{and}\quad \|\partial_x h\|_2^2\ge c\,\|h\|_2^2\ .
\eeqn
Hence, writing
$$
m=\frac{1}{\beta+h_*}\big(\beta\,m-h_*\,\Gamma+h_*\,(m+\Gamma)\big)\ ,\qquad \Gamma=\frac{1}{\beta+h_*}\big(\beta\, (m+\Gamma)-(\beta\,m-h_*\,\Gamma)\big)
$$
we obtain from \eqref{63} that
$$
\| m\|_2^2+\|\Gamma\|_2^2\le c\,\left(\|\beta\, m-h_*\Gamma\|_2^2+\|\partial_x m\|_2^2+\|\partial_x\Gamma\|_2^2\right)\ .
$$
Plugging this into \eqref{61} and making $\eps$ (and hence $m_*\le\eta_*<\eps$ on the right-hand side of \eqref{61}) smaller if necessary, we deduce that
\bqnn
\begin{split}
\frac{1}{2}\,\frac{\rd }{\rd t}\left(q\,\|h\|_2^2+\beta\, \|m\|_2^2+h_*\,\|\Gamma\|_2^2\right)& +\mu_0\left(q\,\|h\|_2^2+\beta\, \|m\|_2^2+h_*\,\|\Gamma\|_2^2\right)\le 0
\end{split}
\eqnn
with a sufficiently small number $\mu_0:=\mu_0(h_*)>0$. This, in turn, readily implies that $w=(h,m,\Gamma)$ has exponential decay, i.e.,
$$
\left\|e^{-t A_*}w^0\right\|_2\le M\,e^{-\omega_0 t}\,\|w^0\|_2\ ,\quad t\ge 0\ ,
$$
for some numbers $M:=M(h_*)\ge 1$ and $\omega_0:=\omega_0(h_*)>0$. This yields the assertion.
\end{proof}

\medskip

Finally, note that the function $F$ on the right-hand side of \eqref{52a} is defined on any sufficiently small neighborhood $\mathcal{V}$ of zero in $PH_N^2$ and that, since $b(u_*)u_*=0$, 
\beqn\label{65}
F\in C^{2-}(\mathcal{V},PL_2)\ ,\quad F(0)=0\ ,\quad F'(0)=0\ ,
\eeqn
with $F'$ denoting the Fr\'echet derivative of $F$. The smallness condition on $\mathcal{V}$ is needed to guarantee that the first component of $u_*+v$ does not vanish so that $b(u_*+v)$ is well-defined and we may in particular choose the zero neighborhood $\mathcal{V}$ in $PH_N^2$ such that
\beqn\label{66}
\|v\|_\infty<\min\{h_*,m_*,\Gamma_*\}\ ,\quad v\in \mathcal{V}\ .
\eeqn
Lemma~\ref{L1}, Lemma~\ref{L2}, and \eqref{65} are then the ingredients to apply the principle of linearized stability \cite[Thm.~9.1.2]{Lunardi} to \eqref{52a}, \eqref{52b}. We thus conclude that given any $\omega\in (0,\omega_0)$ there exists $R>0$ such that for any initial value $v^0$ belonging to a ball in $PH_N^2$ centered at zero of sufficiently small radius, the problem \eqref{52a}, \eqref{52b} admits a unique global solution $v\in C^1([0,\infty),PL_2)\cap C([0,\infty),PH_N^2)$ with $v(t)\in \mathcal{V}$ for $t\ge 0$ and
$$
\|v(t)\|_{PH_N^2}+\|\dot{v}(t)\|_{PL_2}\le R\, e^{-\omega t}\, \|v^0\|_{PH_N^2}\ ,\quad t\ge 0\ .
$$
Therefore, $u:=[t\mapsto v(t)+u_*]$ is the unique solution to \eqref{30}. Notice that \eqref{66} and $v(t)\in \mathcal{V}$ for $t\ge 0$ warrant the positivity of $u$. 

Summarizing we have shown the following result on the local asymptotic stability of steady states for \eqref{1A}-\eqref{4AA} with small surfactant concentration:

\begin{thm}\label{T2}
Let $\sigma\in C^2(\R)$ be decreasing. Let $h_*>0$ be arbitrary. Then there exist numbers \mbox{$\eps:=\eps(h_*)>0$}, $\omega:=\omega(h_*)>0$, and $M:=M(h_*)>0$ such that for
$$
m_*:=\frac{h_*}{\beta+h_*}\, \eta_*\ ,\quad \Gamma_*:=\frac{\beta}{\beta+h_*}\, \eta_*\ 
$$
with $\eta_*\in (0,\eps)$ and any initial value $(h^0,m^0,\Gamma^0)\in H_N^2$ with $\langle h^0\rangle =h_*$ and $\langle m^0+\Gamma^0\rangle=\eta_*$ satisfying the smallness condition
$$
\|h^0-h_*\|_{H^2}+\|m^0-m_*\|_{H^2}+\|\Gamma^0-\Gamma_*\|_{H^2}\le \eps\ ,
$$
there is a unique global positive solution
$
(h,m,\Gamma)\in C^1([0,\infty),L_2)\cap C([0,\infty),H_N^2)
$
to \eqref{1A}-\eqref{4AA}. This solution satisfies
\bqnn
\begin{split}
\|h(t)-h_*\|_{H^2}\,+\,&\|m(t)-m_*\|_{H^2}\,+\,\|\Gamma (t)-\Gamma_*\|_{H^2}\\
& \le M\,e^{-\omega t}\,\big(\|h^0-h_*\|_{H^2}+\|m^0-m_*\|_{H^2}+\|\Gamma^0-\Gamma_*\|_{H^2}\big)
\end{split}
\eqnn
for $t\ge 0$.
\end{thm}




\end{document}